\documentclass[english,reqno]{amsart}

\usepackage{latexsym}
\usepackage{amssymb,amsbsy,amsmath,amsfonts,amssymb,amscd}
\usepackage{mathrsfs}

\usepackage{float}
\usepackage[caption = false]{subfig}
\usepackage[final]{graphicx}

\usepackage{subfig}
\usepackage{graphicx}

\usepackage{color}

\usepackage{cases}
\theoremstyle{plain}

\newtheorem{theorem}{Theorem}

\theoremstyle{definition} % For roman text in the body
\newtheorem{definition}{Definition}
\newtheorem{remark}{Remark}

\renewcommand{\vec}[1]{\mathbf{#1}}
\newcommand{\eps}{\varepsilon}
\newcommand{\average}[1]{ \left\langle#1 \right\rangle}

\newcommand{\NullL}{{\rm Null} \,\opL}

\newcommand{\Amat}{\mathsf{A}}

\newcommand{\opL}{\mathcal{L}}
\newcommand{\dis}{\text{dis}}

\newcommand{\tsigmas}{\tilde{\sigma}_s}
\newcommand{\tsigmasn}{\tilde{\sigma}_{sn}}

\newcommand{\rd}{\mathrm{d}}

\newcommand{\RR}{\mathbb{R}}
\newcommand{\Sp}{\mathbb{S}}
\newcommand{\Kn}{\mathsf{Kn}}

\newcommand{\inner}{\frac{1-\mu}{\epsilon}}

\title{Stability of inverse transport equation in diffusion scaling and Fokker-Planck limit}
\author{Ke Chen} 
\address{Mathematics Department, University of Wisconsin-Madison, 480 Lincoln Dr., Madison, WI 53705 USA.}
\email{ke@math.wisc.edu}
\author{Qin Li} 
\address{Mathematics Department and Wisconsin Institute of Discovery, University of Wisconsin-Madison, 480 Lincoln Dr., Madison, WI 53705 USA.}
\email{qinli@math.wisc.edu}
\author{Li Wang} 
\address{Department of Mathematics, Computational and Data-Enabled Science and Engineering Program, State University of New York at Buffalo, 244 Mathematics Building, Buffalo, NY 14206 USA.}
\email{lwang46@buffalo.edu}
\thanks{The work of K.C. and Q. L. is supported in part by a start-up fund of Q.L. from UW-Madison and National Science Foundation under the grant DMS-1619778 and DMS-1107291: RNMS KI-Net. The work of L.W. is supported in part by a start-up fund from SUNY Buffalo and the National Science Foundation under the grant DMS-1620135. Both Q.L. and L.W. are grateful to Prof. Kui Ren's inspiring discussions.}

\begin{document}
\maketitle

\begin{abstract}
We consider the inverse problem of reconstructing the scattering and absorption coefficients using boundary measurements for a time dependent radiative transfer equation (RTE). As the measurement is mostly polluted by errors, both experimental and computational, an important question is to quantify how the error is amplified in the process of reconstruction. In the forward setting, the solution to the RTE behaves differently in different regimes, and the stability of the inverse problem vary accordingly. In particular, we consider two scalings in this paper. The first one concerns with a diffusive scaling whose macroscopic limit is a diffusion equation. In this case, we showed, following the similar approach as in [Chen, Li and Wang, arXiv:1703.00097], that the stability degrades when the limit is taken. The second one considers a highly forward peaked scattering, wherein the scattering operator is approximated by a Fokker-Planck operator as a limit. In this case, we showed that a fully recover of the scattering coefficient is less possible in the limit, whereas obtaining a rescaled version of the scattering coefficient becomes more practice friendly. 
\end{abstract}

%%%%%%%%%%%%%%%%%%%%%%%%%%%%%%%%%%%%%%%%%%%%%%%%%%
\section{Introduction}
Radiative transfer equation (RTE) describes the dynamics of photon particles propagating in scattering and absorbing media~\cite{CaseZweifel}. A typical form reads as 
\begin{equation} \label{eqn:000}
\partial_t f + v\cdot\nabla_xf = \int k(x,v,v')f(x,v')\rd{v'} - \sigma(x,v) f\,,
\end{equation}
equipped with a Dirichlet boundary condition
\begin{equation} \label{eqn:BC}
f|_{\Gamma_-} = \phi(t,x,v)\,.
\end{equation}
Here $f(t, x,v)$ is the distribution of particles at location $x \in\Omega\subset\RR^d$ moving with velocity $v \in \Sp^{d-1}$. Since photons travel with a fixed speed, the velocity $v$ is normalized to $|v|=1$. $k(x,v,v')$ is the scattering cross section, representing the probability of particles that move in direction $v'$ changing to direction $v$. $\sigma$ is the total scattering coefficient that consists of the amount of photon particles being scattering and absorbed by the material. $k$ and $\sigma$ constitute the main optical property of the material.

The boundary condition \eqref{eqn:BC} is a common choice for RTE, and $\Gamma_-$ represents the ``incoming" portion of the boundary, i.e., 
\begin{equation}\label{eqn:boundary}
\Gamma_- = \{(x,v): x\in\partial\Omega\,, ~  v\cdot n_x <0\}\,,
\end{equation} 
where $n_x$ is the unit outer normal direction of the boundary. Similarly, one can define the ``outgoing" portion of the boundary by
\begin{equation}\label{eqn:boundary2}
\Gamma_+ = \{(x,v): x\in\partial\Omega\,, ~  v\cdot n_x >0\}\,.
\end{equation} 
The well-posedness of the forward problem \eqref{eqn:000} \eqref{eqn:BC} is summarized in~\cite{Lions93}.

RTE \eqref{eqn:000} is often incorporated with different scales that lead it to different equations. One typical scaling is the diffusive scaling, under which the RTE is well approximated by a diffusion equation 
\begin{equation*}
\partial_t\rho = C\nabla_x\left(\frac{1}{\sigma_s}\nabla_x\rho\right) + \sigma_a \rho\,,
\end{equation*}
where $\rho(t,x) = \int_{\mathbb{S}^{d-1}}f \rd v$, and $\sigma_s$, $\sigma_a$ related to $k$ and $\sigma$ will be defined later. $C$ is a generic constant depending on the dimension of the problem. This scaling is encountered in the long time limit with a strong scattering effect. Another is the Fokker-Planck scaling which emphasizes the highly forward peaked scattering. In this case, \eqref{eqn:000} reduces to 
\begin{equation*}
\partial_t f + v\cdot\nabla_xf = \mathcal{L}_\text{FP}f\,,
\end{equation*}
where $$\mathcal{L}_\text{FP} = \left[   \frac{\partial}{\partial v_3} (1-v_3^2) \frac{\partial}{\partial v_3} + \frac{1}{1-v_3^2} \frac{\partial^2}{ \partial \psi^2}   \right] \,,$$ and $v = (\sqrt{1-v_3^2}\cos\psi,\sqrt{1-v_3^2}\sin\psi,v_3)$. In both scenarios, theory exits regarding the derivation, validity, and asymptotic error in the approximation, the reader can make references to \cite{LK74, BSS84} for the former case and \cite{Larsen_FP, Pomraning_FP} for the latter. 

We study the inverse problem in this paper, with special attention paid to how its stability varies under the above two scalings. Unlike the forward setting wherein the optical properties $k$ and $\sigma$ are given, and one amounts to solve $f(t,x,v)$ for a specific boundary condition~\eqref{eqn:BC}, in the inverse problem setting, one tries to recover the unknown optical properties from boundary measurements of $f(t,x,v)$. To be more precise, we define the albedo operator as a mapping from the boundary condition $\phi(t,x,v)$ to the outgoing data $f|_{\Gamma_+}$:
\begin{equation*}
\mathcal{A}(k,\sigma):\quad\phi\mapsto f|_{\Gamma_+}\,,
\end{equation*}
then by adjusting the incoming data $\phi$, and measuring the corresponding outgoing data $f|_{\Gamma_+}$, one gains a full knowledge of $\mathcal{A}$, which can be used to determine $k$ and $\sigma$.

The inverse RTE problem benefits a broad application in optical tomography, atmospheric science and aerospace engineering. Optical tomography, with its major application in medical imaging, utilizes scattered light as a probe of structural variations in the optical properties of the tissue. Specifically, a narrow collimated beam of low energy visible or near infrared light is sent into biological tissues, and then collected by an array of detectors after it propagates through the media. The measurements collected are used to recover the optical properties of the media.  In atmospheric science, or remote sensing, satellites cumulate hyperspectral light reflected from the earth and is used to infer mineral or plant distribution on the ground. In aerospace engineering, pictures taken by spacecrafts in the universe (Galileo's pictures from Jupiter, or Cassini's pictures from Saturn for example) are sent back to the earth for analyzing mineral/gas distribution on different planets. In all the applications, the forward solver for the light propagation is described by the RTE. One measures the reflected or propagated light intensity to reconstruct the optical properties, with which tissue/ground/gas components are inferred.

On the analytical side, there has been a vast literature on the wellposedness and stability of the inverse problem. In a pioneering paper~\cite{ChoulliStefenov}, the authors showed that both $k$ and $\sigma$ can be uniquely determined by the incoming-to-outgoing map $\mathcal{A}$, assuming that $\sigma$ is $v$-independent. With $v$ dependence, the uniqueness up to the gauge-invarience was shown in~\cite{Stefenov2}. The analysis is done through performing the singular decomposition: one separates the collected $f|_{\Gamma_+}$ data according to the singularities, and different parts are in charge of recovering different coefficients. Another approach is to linearize the equation before applying inverse Born series, and followed by showing the convergence of the series~\cite{Machida_Schotland_inverse}. The results on the stability of the ``inverse" dates back to~\cite{Wang} and was made systematic in~\cite{Bal09, Bal10a, Bal08}. Many papers concern the time-dependent case and the associated stability analysis has also been conducted~\cite{Larsen88, Romanov96, CS96}, and also \cite{Bal_review} for a review.

On the numerical side, special care is needed to address for the illposedness of the problems, both inherited from the continuous counterpart, and due to the incomplete corrupted data. Indeed, to uniquely determine $k$ and $\sigma_a$, one needs a full knowledge of $\mathcal{A}$, which is impractical in real applications; and measurement error can easily propagate and get exaggerated. Typically Tikhonov type regularization is used to balance the pollution and the error tolerance, and the type of regularization embeds some prior knowledge. See~\cite{SRH05} using the standard $L_2$,~\cite{Tang} using TV regularization for the least variance,~\cite{Ren_review} summarizing $H_1$ regularization for some regularity, and $L_1$ regularization for sparsity. See also Tikhonov type regularization used on each element in the inverse Born series~\cite{Machida:16,Machida_Schotland_inverse}. Besides the illposedness, the size of the problem also brings extra difficulties, and Jacobian-type techniques~\cite{Arridge_gradient_2013} are introduced to advance the computation.

In the presence of different scales, however, the above mentioned theory or algorithms cannot be directly applied since the inverse problem may completely change its type. One example is the diffusive regime, while the inverse RTE with sufficient variation in measurement is shown to be well-posed, its diffusion limit is the Calder\'on type problem which is well acknowledged to be ill-posed \cite{Uhlmann03,Uhlmann09}. Our goal in this paper is to provide a rigorous connection between different scalings in the inverse setting, and show how stability varies with the scaling parameter. For the diffusive scaling, the connection is observed in ~\cite{RenBalHielscher_trans_diff, Arridge99, Arridge98} and addressed in ~\cite{Arridge_Schotland09,ChenLiWang}. A similar problem on recovering the doping profile in the Boltzmann-Poisson system is presented ~\cite{Gamba}, wherein numerical simulations also implies this relation. In the Fokker-Planck regime, the limit was briefly mentioned in~\cite{Bal_review} but the full discussion was rarely seen in the literature. The main contributions in this paper are
\begin{itemize}
\item[1)] extend our previous analysis with diffusive scaling for steady problems \cite{ChenLiWang} to time-dependent problems, and show that the stability degrades in the diffusion limit;
\item[2)] examine the well-posedness and stability in the Fokker-Planck scaling which has never been studied in detail before.
\end{itemize}

The rest of paper is organized as follows. Section 2 is devoted to diffusion regime and Section 3 is devoted to forward peaked regime, in which diffusion equation and Fokker-Planck equation are obtained as asymptotic limit respectively. In both cases, we utilize the linearization approach, study the well-posedness of the problem in both regimes, and examine the change of stability while passing the limit.

%%%%%%%%%%%%%%%%%%%%%%%%%%%%%%%%%%%%%%%%%%%%%%%%%%%%%%
\section{Diffusion regime}
In this section, we study the wellposedness of the inverse RTE in the diffusion regime. First we briefly recapitulate the properties of RTE and its diffusion limit. For the ease of notation, we assume that the optical properties only have spatial dependence, and rewrite \eqref{eqn:000} as
\begin{equation} \label{eqn:001}
\partial_t f + v\cdot\nabla_xf = \sigma_s(x) \int  (f(t,x,v')-f(t,x,v)) \rd v' - \sigma_a(x) f(t,x,v) \,,
\end{equation}
where $\sigma_s$ is termed the scattering coefficient, and 
\begin{equation*}
\sigma_a(x) = \sigma(x)-\sigma_s(x) 
\end{equation*}
is the absorption coefficient. 

The diffusion regime is achieved in the long time limit and when the scattering is much stronger than the absorption. To this end, we introduce a small parameter---Knudsen number $\Kn$ and rescale the RTE as follows:
\begin{equation}\label{eqn:RTE_s}
\begin{cases}
\Kn \partial_t f+v\cdot \nabla_x f= \frac{1}{\Kn}\sigma_s \mathcal{L}f-\Kn \sigma_a f\quad \text{in}\quad (0,T)\times \Omega \times \mathbb{S}^{n-1} \,,\\
f(0,x,v)=f^I(x,v) \,, \\
f|_{\Gamma_-}(t,x,v)=\phi(t,x,v)  \,.
\end{cases}
\end{equation}
Here the collision operator $\opL$ is an abbreviation of 
\begin{equation}\label{eqn:collision}
\mathcal{L} f (t,x,v)= \int  (f(t,x,v')  - f(t,x,v) )\rd v' \,.
\end{equation}
There are two key features of the collision operator:
\begin{itemize}
\item{Mass conservation}: $\int \mathcal{L}f\rd{v} = 0$.
\item{One dimensional Null space}: By setting $\mathcal{L}f=0$, one gets $f = \int f \rd v$, meaning that $f$ is a constant in velocity domain. We denote it as $\NullL = \text{span}\{\rho(t,x)\}$, the collection of functions that depend on $t$ and $x$ only.
\end{itemize}

%%%%%%%%%%%%%%%%%%%%%%%%%%%%%%%%%%%%%%%%%%%%%%%%%%%%%%
\subsection{Diffusion limit}
When $\Kn\ll 1$, the equation falls into the diffusion regime, and the RTE is asymptotically equivalent to a diffusion equation as $\Kn \rightarrow 0$.
\begin{theorem}\label{thm:diffusion}
Suppose $f$ solves~\eqref{eqn:RTE_s} with initial data $f(0,x, v)=f^I(x)$ and boundary data $f|_{\Gamma_-}=\phi(t,x)$, both of which are independent of velocity $v$. Then as $\Kn\to 0$, $f(t,x,v)$ converges to $\rho(t, x)$, which solves the heat equation:
\begin{equation}\label{eqn:diff}
\begin{cases}
\partial_t\rho -C\nabla_x\cdot \left(\frac{1}{\sigma_s}\nabla_x\rho \right) + \sigma_a\rho= 0\,,\\
\rho(0,x)=f^I(x)   \,, \\
\rho|_{\partial\Omega} = \phi(t,x) \,.
\end{cases}
\end{equation} 
Here $C$ is a time dependent constant. 
\end{theorem}

\begin{proof}
The proof follows standard asymptotic expansion. In the zero limit of $\Kn$, the distribution converges to the local equilibrium, and by applying the standard asymptotic expansion technique, we write:
\begin{equation}
f_\text{in} = f_0 + \Kn f_1 + \Kn^2 f_2 +\cdots\,.
\end{equation}
Inserting the expansion in the equation~\eqref{eqn:RTE_s} and equate like powers of $\Kn$:
\begin{itemize}
\item[$\mathcal{O}(1)$] $\mathcal{L}f_0 = 0$. This immediately indicates that $f_0\in\NullL$. With the form given in~\eqref{eqn:collision}, $\NullL$ consists functions that are constants in $v$ domain, and thus $f_0(t,x,v)=\rho(t,x)$.
\item[$\mathcal{O}(\Kn)$] $v\cdot\nabla_xf_0 = \sigma_s\mathcal{L}f_1$. This indicates that $f_1 = \mathcal{L}^{-1}\frac{v\cdot\nabla_xf_0}{\sigma_s}$. Notice that $\mathcal{L}$ is invertible on $\NullL^\perp$, and consider the form of $\mathcal{L}$ in~\eqref{eqn:collision}, then $\NullL^\perp = \{f:\int f\rd{v} = 0\}$ and $\frac{v\cdot\nabla_xf_0}{\sigma_s}\in\NullL^\perp $, therefore $f_1 = \mathcal{L}^{-1}\frac{v\cdot\nabla_xf_0}{\sigma_s}=-\frac{v\cdot\nabla_xf_0}{\sigma_s}$.
\item[$\mathcal{O}(\Kn^2)$] $\partial_t f_0 +v\cdot\nabla_xf_1 = \sigma_s\mathcal{L}f_2 - \sigma_af_0$. Here we integrate the equation with respect to $v$. The second corrector $f_2$ will vanish and the left hand side becomes:
\begin{equation}
\partial_t \rho +\int v\cdot\nabla_x\left(-\frac{v}{\sigma_s}\cdot\nabla_x\rho\right) \rd{v}= -\sigma_a\rho \quad\Rightarrow\quad \partial_t\rho -C\nabla_x\cdot\left(\frac{1}{\sigma_s}\nabla_x\rho\right) = -\sigma_a\rho\,.
\end{equation}
\end{itemize}

Integrating $\int v\cdot v\rd{v}$ out, we obtains the diffusion limit and conclude the theorem. The constant $C$ depends on the dimension of the velocity space.
\end{proof}

\begin{remark} \label{remark:BL00}
We comment that the initial and boundary conditions in the above theorem are relatively strict: they both need to be independent of $v$. If not, one needs to introduce the initial layer and the boundary layer to damp out the nonhomogeneities. More specifically, we write:
\begin{equation}
f(t,x,v) \approx f_\text{il}(t,x,v) + f_\text{bl}(t,x,v) + f_\text{int}(t,x,v)\,,
\end{equation}
where $f_\text{int}$ stands for the interior solution and writes:
\begin{equation*}
f_\text{int}(t,x,v) = \theta(t,x)-\Kn \mathcal{L}^{-1}(v)\partial_x\theta(t,x)\,,
\end{equation*}
with $\theta$ satisfying the diffusion equation~\eqref{eqn:diff}. $f_\text{il}$ is the initial layer and is governed by:
\begin{equation*}
\partial_\tau f_\text{il} - \mathcal{L}f_\text{il} = 0\,,
\end{equation*}
where $\tau = t/\Kn^2$ is the rescaled time. With appropriate initial data, $f_\text{il}$ damps to $0$ exponentially fast in $\tau$ and thus $f_\text{il}\sim e^{-t/\Kn^2}\sim 0$ for finite $t$ with $\Kn\to 0$. $f_\text{bl}$ is the boundary layer. At each point on the boundary, $x_0\in\partial\Omega$, $f_\text{bl}$ satisfies:
\begin{equation*}
v\partial_zf_\text{bl}+\mathcal{L}f_\text{bl}= 0\,,
\end{equation*}
where $z$ is the rescaled spatial coordinate around $x_0$: $z=-\frac{(x-x_0)\cdot n_x}{\Kn}$ with $n_x$ being the normal direction pointing out of $\Omega$ at $x_0$. It has been shown that $f_\text{bl}$ exponentially decays to a constant in $z$. This constant is termed the extrapolation length, and is uniquely determined by the boundary data around $x_0$. We denote it $\phi(t,x_0)$. This means that for $x$ adjacent to $x_0$, in the zero limit of $\Kn$, $|f_\text{bl}-\phi(t,x_0)|\sim e^{-|x-x_0|/\Kn}\sim 0$. We typically subtract this constant from $f_\text{bl}$ and set it as the Dirichlet boundary condition for $\theta$, and thus $f_\text{bl}\sim 0$ everywhere.
\end{remark}

%%%%%%%%%%%%%%%%%%%%%%%%%%%%%%%%%%%%%%%%%%%%%%%%%%%%%%
\subsection{Recover absorption coefficient $\sigma_a$}
In this section we assume that the scattering coefficient is known and aim to recover the absorption coefficient. Without loss of generality, we let $\sigma_s \equiv 1$.\subsubsection{Inverse problem setup}
We first rewrite~\eqref{eqn:RTE_s} into:
\begin{equation}
\begin{cases}\label{eqn:abs}
\Kn \partial_t f+v\cdot \nabla_x f= \frac{1}{\Kn}\mathcal{L}f-\Kn \sigma_a f\quad \text{in}\quad (0,T)\times \Omega \times \mathbb{S}^{n-1}\,,\\
f(0,x,v)=0 \quad \text{on}\quad \{t=0\}\times \Omega\times  \mathbb{S}^{n-1}\,,\\
f(t,x,v)=\phi(t,x,v)\quad \text{on} \quad (0,T)\times \Gamma_-\,.
\end{cases}
\end{equation}
The solution to the above equation, denoted by $f(t,x,v;\phi)$, models the number density of photons with certain inflow $\phi$. In experiment, time-dependent velocity-averaged data $m(t,x):=\int_{\mathbb{S}^{n-1}}v\cdot n(x)f(t,x,v)|_{\Gamma_+}\rd{v}$ is collected on the out-flow boundary $\Gamma_+$. Therefore, we can define the albedo operator as:
\begin{equation*}
\mathcal{A}(\sigma_a): \quad \phi(t,x,v)\rightarrow m(t,x)=\int_{\mathbb{S}^{n-1}}v\cdot n(x)f(t,x,v) |_{\Gamma_+}\rd{v}\,.
\end{equation*}
Notice that $\mathcal{A}$ nonlinearly depends on $\sigma_a$ via the solution $f(t,x,v;\phi)$. In an effort to study the property of $\mathcal{A}$, we first derive a linearized version of it following the procedure outlined in \cite{Ren_review}. 

Suppose a priori information about the absorption coefficient is known in the sense that $\sigma_a$ can be considered as a small perturbation around a background state $\sigma_{a0}(x)$, i.e., 
\begin{equation*}
\sigma_a(x)=\sigma_{a0}(x)+\tilde{\sigma}_a(x)\quad \text{with}\quad |\tilde{\sigma}_a| \ll |\sigma_a|\,,\quad \text{a.s}\,,
\end{equation*}
then a linearized problem with background state $\sigma_{a0}$ and same initial and boundary data can be defined as
\begin{equation}\label{eqn:abs0}
\begin{cases}
\Kn \partial_t f_0+v\cdot \nabla_x f_0= \frac{1}{\Kn}\mathcal{L}f_0-\Kn \sigma_{a0} f_0\quad \text{in}\quad (0,T)\times \Omega \times \mathbb{S}^{n-1}\,,\\
f_0(0,x,v)=0 \quad \text{on}\quad \{t=0\}\times \Omega\times  \mathbb{S}^{n-1}\,,\\
f_0(t,x,v)=\phi(t,x,v)\quad \text{on} \quad (0,T)\times \Gamma_-\,.
\end{cases}
\end{equation}
Comparing \eqref{eqn:abs} and \eqref{eqn:abs0}, we define the residue $\tilde{f}=f-f_0$, then it solves, to the leading order:
\begin{equation}\label{eqn:abs_res}
\begin{cases}
\Kn \partial_t \tilde{f}+v\cdot \nabla_x \tilde{f}= \frac{1}{\Kn}\mathcal{L}\tilde{f}-\Kn \sigma_{a0} \tilde{f}-\Kn\tilde{\sigma}_af_0\quad \text{in}\quad (0,T)\times \Omega \times \mathbb{S}^{n-1}\,,\\
\tilde{f}(0,x,v)=0 \quad \text{on}\quad \{t=0\}\times \Omega\times  \mathbb{S}^{n-1}\,,\\
\tilde{f}(t,x,v)=0\quad \text{on} \quad (0,T)\times \Gamma_-\,,
\end{cases}
\end{equation}
which is obtained by subtracting \eqref{eqn:abs0} from \eqref{eqn:abs} with the higher order term $\tilde{\sigma}_a\tilde{f}$ omitted. Notice here that both the linearized solution $f_0$ and the residue $\tilde{f}$ are implicitly dependent on the incoming data $\phi$. We then introduced an adjoint problem of \eqref{eqn:abs0} and assign a Dirac delta function $\delta(\tau,y)$ at the boundary $ (0,T)\times \Gamma_+$:
\begin{equation}\label{eqn:abs_g}
\begin{cases}
-\Kn \partial_tg-v\cdot \nabla_x g= \frac{1}{\Kn}\mathcal{L}g-\Kn \sigma_{a0} g\quad \text{in}\quad (0,T)\times \Omega \times \mathbb{S}^{n-1}\,,\\
g(T,x,v)=0 \quad \text{on}\quad \{t=T\}\times \Omega\times  \mathbb{S}^{n-1}\,,\\
g(t,x,v)=\delta(\tau,y)\quad \text{on} \quad (0,T)\times \Gamma_+\,.
\end{cases}
\end{equation}
The solution is denoted by $g(t,x,v;\tau,y)$. Multiplying \eqref{eqn:abs_res} with $g$, \eqref{eqn:abs_g} with $f_0$, integrating over $(0,T)\times\Omega\times\mathbb{S}^{n-1}$ and then subtracting them, we get:
\begin{equation}\label{eqn:abs_IBP}
\int_{\Gamma_+(y)}\tilde{f}(\tau,y,v)n(y)\cdot v \rd{v}=-\Kn \int_\Omega \tilde{\sigma}_a(x) \int_{\mathbb{S}^{n-1}}\int_0^T f_0(t,x,v;\phi) g(t,x,v;\tau,y)\rd{t}\rd{v} \rd{x}\,.
\end{equation}
We denote the LHS of \eqref{eqn:abs_IBP} by $b(\tau,y,\phi)$, then according to the definition of $\tilde{f}$, it is simply
\begin{equation}\label{eqn:abs_b}
b(\tau,y,\phi):= \int_{\Gamma_+(y)}f(\tau,y,v)n(y)\cdot v \rd{v}-\int_{\Gamma_+(y)}f_0(\tau,y,v)n(y)\cdot v \rd{v} \,,
\end{equation}
with the first term being the measurement from experiments, and the second term computed from equation~\eqref{eqn:abs0}. This term, therefore is known ahead of time. The RHS of \eqref{eqn:abs_IBP} defines a linear mapping of $\tilde{\sigma}_a$. Let us denote
\begin{equation}\label{eqn:abs_gamma}
\gamma_{\Kn}(x;\tau,y,\phi):=-\Kn \int_{\mathbb{S}^{n-1}}\int_0^T f_0(t,x,v;\phi) g(t,x,v;\tau,y)\rd{t}\rd{v}, 
\end{equation} 
then \eqref{eqn:abs_IBP} defines a family of linear mapping from $\gamma_{\Kn}$ to the data on the LHS, parametrized by $(\tau,y,\phi)$:
\begin{equation}\label{eqn:abs_linear}
\int_\Omega \tilde{\sigma}_a(x)\gamma_{\Kn}(x;\tau,y,\phi)\rd{x} = b(\tau,y,\phi)\,.
\end{equation}
Therefore, \eqref{eqn:abs_linear} defines a linearized albedo operator, from which $\tilde{\sigma}_a$ can be obtained via solving a system of linear equations. 

\begin{remark}\label{rmk:wellposedness_ab}
Equation~\eqref{eqn:abs_linear} is a first type Fredholm operator, and it holds true for all parameter choices of $\tau$, $y$ and $\phi$. The study on the wellposedness simply relies on the space expanded by $\{\gamma_\Kn\}$. Suppose we look for $\tilde{\sigma}_a\in L_p(\rd{x})$, then the uniqueness is guaranteed if $\{\gamma_\Kn\}$ expands $L_q$ space (with $\frac{1}{p} + \frac{1}{q} = 1$). There has been many studies on the topic and is not the main goal of the current paper. The wellposedness amounts to analyze the ``conditioning" of $\gamma_\Kn$. It is closely related to studying its ``singular values", as will be explained in better details below.
\end{remark}

\subsubsection{Ill-conditioning in the diffusion limit}
Given the linearized albedo operator defined in \eqref{eqn:abs_linear}, studying the stability of recovering $\tilde{\sigma}_a$ boils down to examining the property of the Fredholm operator of the first kind defined there. In this section, we intend to explore its conditioning with respect to $\Kn$. More precisely, given a family of input-measurement pairs $(\phi(t,x,v), ~m(\tau, y, \phi))$, where $(t,x,v)\in (0,T)\times \Gamma_-$ and $m(\tau,y,\phi)=\int_{\mathbb{S}^{n-1}}f(\tau,y,v)n(y)\cdot v\rd{v}$, we can explicitly compute $\gamma_{\Kn}(x)$ and $b(\tau,y,\phi)$ defined in \eqref{eqn:abs_gamma} and \eqref{eqn:abs_b}, and study their dependence on $\Kn$ so as to get a sensitivity in recovering $\tilde{\sigma}_a$ with respect to $\Kn$.

In this regard, we first introduce a distinguishability coefficient to quantify the perturbation of $\tilde{\sigma}_a$ when a $\delta$-error is allowed for $b(\tau,y,\phi)$.
\begin{definition}
Consider linear equations \eqref{eqn:abs_linear} and $\gamma_{\Kn}$ defined in \eqref{eqn:abs_gamma} and $b(\tau,y,\phi)$ defined in \eqref{eqn:abs_b}, we defined the distinguishability coefficient as
\begin{equation} \label{eqn:kappa}
\kappa_a:=\sup_{\sigma_a\in\Gamma_\delta} \frac{\|\sigma_a-\tilde{\sigma}_a\|_{L^\infty(\rd x)}}{ \|\tilde{\sigma}_a\|_{L^\infty(\rd x)}} \,,
\end{equation}
where 
\[
\Gamma_\delta=\{ \sigma_a: \sup_{\substack{\forall \|\phi\|_{L^\infty(\Gamma_-)}\leq 1,\\ \forall y\in \partial\Omega, ~ \tau\in [0,T]}}|\langle\gamma_\Kn\,,\sigma_a\rangle_{L^2(\rd{x})} - b(\tau, y,\phi) | \leq \delta  \}\,,
\]
and $\tilde{\sigma}_a$ is the exact solution to \eqref{eqn:abs_linear}. 
\end{definition}
Here $\Gamma_\delta$ consists of all possible solutions to \eqref{eqn:abs_linear} within $\delta$-tolerance, and the distinguishability coefficient $\kappa$ quantifies supremum of relative error over $\Gamma_\delta$. Therefore, in practice small $\kappa$ is desired. However, this is not the case when $\Kn$ is small, as will be shown in the following theorem: small $\kappa$ leads to very bad distinguishability.

\begin{theorem}
For a family of linear equations defined in \eqref{eqn:abs_linear} and an error tolerance $\delta>0$ on the measurement, the distinguishability coefficient satisfies
\[
\kappa_a = \mathcal{O}\left(\frac{\delta}{\Kn} \right)\quad \text{when}\quad \Kn \ll 1\,.
\]
\end{theorem}

\begin{proof}
For any $\sigma_a\in \Gamma_\delta$, define $c=\sigma_a-\tilde{\sigma}_a$, then we have
\begin{equation}\label{eqn:abs_c}
\left| \int_\Omega \gamma_{\Kn}(x)c(x)\rd{x} \right|\leq \delta\,.
\end{equation}
When $\Kn \ll 1$, from Theorem \ref{thm:diffusion}, $f_0(t,x,v)$ can be decomposed into two part 
\[
f_0(t,x,v)=f_\text{int}(t,x,v)+f_\text{bl}(t,x,v)\,,
\]
where $f_\text{bl}(t,x,v)$ encodes the boundary layer supported near the boundary with $\mathcal{O}(\Kn)$ width and $f_\text{int}$ is the interior solution, and it approaches to its diffusion limit $\rho_f(t,x)$ which satisfies \eqref{eqn:diff} with zero initial data and suitable boundary condition. Specifically, $f_\text{int}(t,x,v)$ can be expanded as:
\begin{equation}\label{eqn:abs_exf}
f_\text{int}(t,x,v) = \rho_f(t,x) -\Kn v\cdot \nabla_x \rho_f(t,x) +\mathcal{O}(\Kn^2)\,,
\end{equation}
where $\rho_f$ solves
\begin{equation}\label{eqn:rho_f}
\begin{cases}
\partial_t \rho_f  = C \Delta_x \rho_f -\sigma_a \rho_f     \quad \text{in}\quad (0,T)\times \Omega \,,\\
\rho_f(0,x)=0 \quad \text{on}\quad \{t=0\}\times \Omega  \,,\\
\rho_f(t,x)=\eta_\phi(t,x)  \quad \text{on} \quad (0,T)\times \partial \Omega\,.
\end{cases}
\end{equation}
Here the boundary value $\eta_\phi(x)$ is computed from $\phi(t,x,v)$ through the boundary layer analysis. (Details are provided in Remark~\ref{remark:BL00}). 

Likewise, $g$ admits the same decomposition that separates the interior part from the boundary part:
\[
g_0(t,x,v)=g_\text{int}(t,x,v)+g_\text{bl}(t,x,v)\,,
\]
and $g_\text{int}(t,x,v)$ has the following expansion: 
\begin{equation*}
g_\text{int}(t,x,v) = \rho_g(t,x) -\Kn v\cdot \nabla_x \rho_g(t,x) +\mathcal{O}(\Kn^2)\,,
\end{equation*}
with $\rho_g$ satisfying 
\begin{equation}\label{eqn:rho_f}
\begin{cases}
-\partial_t \rho_g  = C \Delta_x \rho_g -\sigma_a \rho_g     \quad \text{in}\quad (0,T)\times \Omega \,,\\
\rho_g(T,x)=0 \quad \text{on}\quad \{t=T\}\times \Omega \,,\\
\rho_f(t,x)= \eta_\delta (t,x)  \quad \text{on} \quad (0,T)\times \partial \Omega\,.
\end{cases}
\end{equation}

We plug the expansion of $f_0$ and $g$ in the definition of $\gamma_\Kn$, then in the interior away from the layer:
\begin{align*}
\gamma_{\Kn}(x;\tau,y,\phi)&:=-\Kn \int_{\mathbb{S}^{n-1}}\int_0^T f_0(t,x,v;\phi) g(t,x,v;\tau,y)\rd{t}\rd{v}\,,\\
& = -\Kn \int_0^T \rho_f(t,x)\rho_g(t,x)\rd{t} +\mathcal{O}(\Kn^3)\,.
\end{align*} 
Simplification is not available inside the layer. In the derivation, the $\mathcal{O}(\Kn^2)$ terms are:
\begin{equation*}
\int_0^T\rho_f\int_\mathbb{S}^{n-1} v\cdot\nabla_x\rho_g\rd{v}\rd{t} + \int_0^T\rho_g\int_\mathbb{S}^{n-1} v\cdot\nabla_x\rho_f\rd{v}\rd{t}\,,
\end{equation*}
and they disappear since the integrands are odd functions. Then inserting this $\gamma_{\Kn}$ back into \eqref{eqn:abs_c} we have
\begin{equation*}
\int_\Omega c(x)\gamma_{\Kn}(x)\rd{x}=-\Kn \int_\text{int} c(x)\int_0^T\rho_f(t,x)\rho_g(t,x)\rd{t}\rd{x} +\int_\text{bl} c(x)\int_0^T\gamma_{\Kn}(x)\rd{t}\rd{x} +\mathcal{O}(\Kn^3)
\end{equation*} 
which readily implies that $\kappa=\mathcal{O}(\frac{\delta}{\Kn})$ by choosing $c=0$ inside the layer.
\end{proof}
We note that $\rho_f$ and $\rho_g$ are solutions to the heat equation, and have no dependence on $\Kn$. In the time-independent case~\cite{ChenLiWang}, we compensate it by showing that the multiplication of $\rho_f$ and $\rho_g$, the solutions to two elliptic equations are of low rank, which easily produces one more $\Kn$. It is no longer the case here: we cannot prove the term $\int_0^T\rho_f\rho_g\rd{t}$ being low rank in $L_2$, and thus it is hard to obtain $\mathcal{O}(\delta/\Kn^2)$.
 
\begin{remark}
We would like to point out that the definition of distinguishability coefficient \eqref{eqn:kappa} is a ``continuous" analog of the condition number in the discrete setting. In fact, if we discretize equation~\eqref{eqn:abs_linear} in $x$ and write it in a matrix form, we get:
\begin{equation*}
\mathsf{A}\cdot\tilde{\sigma}_a^{\dis} = \mathsf{b}\,,
\end{equation*}
where each row of $\Amat$ is $\gamma_\Kn$ evaluated at all discrete point with one particular $\tau$, $y$ and $\phi$ selected:
\begin{equation*}
\mathsf{A}_{ij} = \gamma_\Kn(x_j;\tau_i,y_i,\phi_i)\,,\quad\text{and}\quad \mathsf{b}_i = b(\tau_i,y_i,\phi_i)\,.
\end{equation*}
Perform the singular value decomposition of $\mathsf{A}$
\begin{equation*}
\mathsf{A} = \mathsf{U}\cdot\Sigma\cdot\mathsf{V}^T = \sum_{i=1}^N\lambda_i \mathsf{u}_i \mathsf{v}^T_i\,,\quad \lambda_1\geq\lambda_2\geq\cdots\geq\lambda_N\,,
\end{equation*}
with $\lambda_i$ being the singular values and $\mathsf{u}_i$ and $\mathsf{v}_i$ are column vectors, then 
\[
\tilde{\sigma}_a^\dis = \mathsf{V} \cdot \Sigma^{-1} \cdot \mathsf{U}^T  \mathsf{b}\,.
\]
Similarly, a variation of $\tilde{\sigma}_a^\dis$, denoted as $\sigma_a^\dis$, satisfies 
\[
\sigma_a^\dis = \mathsf{V} \cdot \Sigma^{-1} \cdot \mathsf{U}^T  (\mathsf{b} + \mathsf{b}^\delta)\,.
\]
Then the equivalent definition of $\kappa$ here is:
\[
\kappa_\mathsf{A} = \max_{\mathsf{b}^\delta: \|\mathsf{b}^\delta\|_{\infty}<\delta} \frac{\|\sum \frac{1}{\lambda_i} \mathsf{v}_i \mathsf{u}_i^T \mathsf{b}^\delta  \|_\infty}{\|\sum \frac{1}{\lambda_i} \mathsf{v}_i  \mathsf{u}_i^T \mathsf{b} \|_\infty} \,.
\]
Assuming $\mathsf{b}$ is a fixed vector, and let the denominator being $\mathcal{O}(1)$, then the biggest number is achieved if $\mathsf{b}^\delta$ is aligned with $\mathsf{u}_N$ so that $\kappa_\mathsf{A} =\frac{\delta}{\lambda_N}$. Or $\kappa_\mathsf{A}$ may implicitly depend on $\mathsf{b}$ as well, and if the definition is replace by:
\[
\kappa_\mathsf{A} = \max_{\mathsf{b}^\delta, \mathsf{b}: \|\mathsf{b}^\delta\|_{\infty}<\delta \|\mathsf{b}\|_\infty} \frac{\|\sum \frac{1}{\lambda_i} \mathsf{v}_i \mathsf{u}_i^T \mathsf{b}^\delta  \|_\infty}{\|\sum \frac{1}{\lambda_i} \mathsf{v}_i  \mathsf{u}_i^T \mathsf{b} \|_\infty} \,,
\]
then the maximum is achieved by the condition number: $\kappa_\mathsf{A} = \frac{\delta\lambda_1}{\lambda_N}$, by aligning $\mathsf{b}$ with $\mathsf{u}_1$ and $\mathsf{b}^\delta$ with $\mathsf{u}_N$.
\end{remark}

\subsection{Recover scattering coefficient $\sigma_s$}
To recover $\sigma_s$ we follow the same route: first set up the inverse problem through a linearization and show that the stability degrades as $\Kn\to 0$. Without loss of generality, we set $\sigma_a = 1$.
\subsubsection{Inverse problem setup}
To set up the inverse problem we first recall the forward problem:
\begin{equation}
\begin{cases}\label{eqn:sca}
\Kn \partial_t f+v\cdot \nabla_x f= \frac{1}{\Kn}\sigma_s\mathcal{L}f-\Kn f\quad \text{in}\quad (0,T)\times \Omega \times \mathbb{S}^{n-1}\,,\\
f(0,x,v) = 0 \quad \text{on}\quad \{t=0\}\times \Omega\times  \mathbb{S}^{n-1}\,,\\
f(t,x,v)=\phi(t,x,v)\quad \text{on} \quad (0,T)\times \Gamma_-\,,
\end{cases}
\end{equation}
then a similar linearization procedure can be conducted as follows. Assume that $\sigma_{s}(x)$ can be written as a superposition of a known background $\sigma_{s0}(x)$ and a perturbation $\tilde{\sigma}_s(x)$ from the background, i.e., 
\begin{equation*}
\sigma_s(x)=\tilde{\sigma}_s(x)+\sigma_{s0}(x)\quad\text{with}\quad |\tilde{\sigma}_s| \ll |\sigma_s| \,,\quad\text{a.s.}\,,
\end{equation*}
then the background solution $f_0$ satisfies equation:
\begin{equation}
\begin{cases}\label{eqn:sca0}
\Kn \partial_t f_0+v\cdot \nabla_x f_0= \frac{1}{\Kn}\sigma_{s0}\mathcal{L}f_0-\Kn f_0\quad \text{in}\quad (0,T)\times \Omega \times \mathbb{S}^{n-1}\,,\\
f_0(0,x,v)= 0  \quad \text{on}\quad \{t=0\}\times \Omega\times  \mathbb{S}^{n-1}\,,\\
f_0(t,x,v)=\phi(t,x,v)\quad \text{on} \quad (0,T)\times \Gamma_-\,.
\end{cases}
\end{equation}
The residue $\tilde{f} := f - f_0$ then solves:
\begin{equation}
\begin{cases}\label{eqn:sca_res}
\Kn \partial_t \tilde{f}+v\cdot \nabla_x \tilde{f}= \frac{1}{\Kn}\sigma_{s0}\mathcal{L}\tilde{f}-\frac{1}{\Kn}\tilde{\sigma}_s\mathcal{L}f_0-\Kn \tilde{f}\quad \text{in}\quad (0,T)\times \Omega \times \mathbb{S}^{n-1}\,,\\
\tilde{f}(0,x,v)=0 \quad \text{on}\quad \{t=0\}\times \Omega\times  \mathbb{S}^{n-1}\,,\\
\tilde{f}(t,x,v)=0\quad \text{on} \quad (0,T)\times \Gamma_-\,.
\end{cases}
\end{equation}
Write the adjoint problem of \eqref{eqn:sca0} as:
\begin{equation}
\begin{cases}\label{eqn:sca_g}
-\Kn \partial_t g-v\cdot \nabla_x g= \frac{1}{\Kn}\sigma_{s0}\mathcal{L}g-\Kn g\quad \text{in}\quad (0,T)\times \Omega \times \mathbb{S}^{n-1}\,,\\
g(T,x,v)=0 \quad \text{on}\quad \{t=T\}\times \Omega\times  \mathbb{S}^{n-1}\,,\\
g(t,x,v)=\delta(\tau,y)\quad \text{on} \quad (0,T)\times \Gamma_+\,,
\end{cases}
\end{equation}
then multiply it with $\tilde{f}$ and subtract the product of \eqref{eqn:sca_res} with $g$, and integrate over $(0,T)\times \Omega \times \mathbb{S}^{n-1}$, we get  
\begin{equation}\label{eqn:sca_IBP}
\int_{\Gamma_+(y)}\tilde{f}(\tau,y,v)n(y)\cdot v \rd{v}=\frac{1}{\Kn} \int_\Omega\tilde{\sigma}_s(x) \int_{\mathbb{S}^{n-1}}\int_0^T g(t,x,v;\tau,y) \mathcal{L}f_0(t,x,v;\phi) \rd{t}\rd{v} \rd{x}\,.
\end{equation}
This equation prompts a linear equation for $\tilde{\sigma}_a$, that is, 
\begin{equation}\label{eqn:sca_linear}
\int_\Omega \tilde{\sigma}_s(x)\gamma_{\Kn}(x;\tau,y,\phi)\rd{x} = b(\tau,y,\phi) \,,
\end{equation}
where 
\begin{eqnarray*}
\gamma_{\Kn}(x;\tau,y,\phi)&:=&\frac{1}{\Kn}\int_{\mathbb{S}^{n-1}}\int_0^T g(t,x,v;\tau,y) \mathcal{L}f_0(t,x,v;\phi) \rd{t}\rd{v} \,, 
\\ &=& \frac{1}{\Kn} \int_0^T \average{g} \average{f} - \average{gf} dt \,,
\end{eqnarray*}
and 
\[
b(\tau,y,\phi):=\int_{\Gamma_+(y)}(f(\tau,y,v) - f_0(\tau, y, v))n(y)\cdot v \rd{v} \,.
\]
Here $\average{f} := \int_{\mathbb{S}^{n-1}} f(t,x,v) \rd v$.
Again $b(\tau,y,\phi)$ is the data at our disposal---the difference between the measured data and computed data, and we end up with a Fredholm operator of first kind with kernel $\gamma_\Kn(x)$. We study in the next section the ill-conditioning for this family of linear equations in the diffusion regime. 

%%%%%%%%%%%%%%%%%%%%%%%%%%%%%%%%%%%%%%%%%%%%%%%
\subsubsection{Ill-conditioning in the diffusion limit}
Similar to the previous case, when $\Kn$ decreases, the transport equation approaches a diffusion equation and thus recovering the scattering coefficient $\sigma_s$ is less stable. More precisely, we have the following theorem:
\begin{theorem}
For a family of linear equations defined in \eqref{eqn:sca_linear} and an error tolerance $\delta>0$ on the measurement, define the distinguishability coefficient as
\begin{equation} \label{eqn:kappa_s}
\kappa_s:=\sup_{\sigma_s\in\Gamma_\delta} \frac{\|\sigma_s-\tilde{\sigma}_s\|_{L^\infty(\rd x)}}{ \|\tilde{\sigma}_s\|_{L^\infty(\rd x)}} \,,
\end{equation}
where 
\[
\Gamma_\delta=\{ \sigma_s: \sup_{\substack{\forall \|\phi\|_{L^\infty(\Gamma_-)}\leq 1,\\ \forall y\in \partial\Omega, ~ \tau\in [0,T]}}|\langle\gamma_\Kn\,,\sigma_s\rangle_{L^2(\rd{x})} - b(\tau, y,\phi) | \leq \delta  \}\,,
\]
and $\tilde{\sigma}_s$ is the exact solution to \eqref{eqn:sca_linear}. Then we have
\[
\kappa_s :=\mathcal{O} \left(\frac{\delta}{\Kn} \right)\quad \text{when}\quad \Kn \ll 1\,.
\]
\end{theorem}

\begin{proof}
The proof again follows a boundary-interior decomposition and asymptotic expansion. First write $g$ and $f_0$ as:
\[
f_0=f_\text{bl}+f_\text{int},\quad g=g_\text{bl}+g_\text{int}
\]
where $f_\text{bl}$ and $g_\text{bl}$ are the boundary layer part,  and $f_\text{int}$ and $g_\text{int}$ are the interior part that admit the following expansion:
\begin{equation}\label{eqn:sca_expand}
\begin{aligned}
f_\text{int}& =\rho_f -\Kn\frac{v\cdot \nabla_x \rho_f}{\sigma_{s0}}+\Kn^2 f_2 \,,\\
g_\text{int}&=\rho_g +\Kn\frac{v\cdot \nabla_x \rho_g}{\sigma_{s0}}+\Kn^2 g_2\,.
\end{aligned}
\end{equation}
Here $\rho_f$ and $\rho_g$ satisfy the diffusion equations:
\begin{equation*}
\partial_t \rho_f +\rho_f -C\nabla_x \left(\frac{1}{\sigma_{s0}}\nabla_x \rho_f \right)=0\,,\quad
\partial_t \rho_g -\rho_g +C\nabla_x \left(\frac{1}{\sigma_{s0}}\nabla_x \rho_g \right)=0\,
\end{equation*}
with suitable initial data and boundary condition.

Now decompose $\gamma_\Kn$ also into a layer and interior parts, i.e., $\gamma_\Kn = (\gamma_\Kn)_\text{bl} + (\gamma_\Kn)_\text{int}$, then for the interior part, using \eqref{eqn:sca_expand}, we have
\begin{eqnarray}
\left( \gamma_\Kn \right)_\text{int} &=& \frac{1}{\Kn} \int_0^T \left( \average{f_\text{int}} \average{g_\text{int}} - \average{f_\text{int} g_\text{int}} \right) \rd{t}  \nonumber
\\ & = & \frac{1}{\Kn} \int_0^T \left( \rho_f + \Kn^2 \average{f_2}\right) \left( \rho_g + \Kn^2 \average{g} \right) \nonumber
\\ && \hspace{1cm} - \average{\left( \rho_f - \frac{\Kn}{\sigma_{s0}} v \cdot \nabla_x \rho_f + \Kn^2 f_2 \right) \left( \rho_g + \frac{\Kn}{\sigma_{s0} }v\cdot \nabla_x \rho_g + \Kn^2 g_2 \right)} \rd{t} \nonumber 
\\ &=& \frac{\Kn}{\sigma_{s0}^2} \int_0^T\average{(v\cdot \nabla_x \rho_f) (v\cdot \nabla_x \rho_g)} \rd{t} + \mathcal{O}(\Kn^2) \nonumber
\\ &=& \frac{C\Kn}{\sigma_{s0}^2} \int_0^T \nabla_x \rho_f \cdot \nabla_x \rho_g \rd{t}+  \mathcal{O}(\Kn^2)\,,
\end{eqnarray}
where $C$ again depends on the dimension of the velocity space. 

 Now denote $c(x)=\sigma_s(x)-\tilde{\sigma}_s(x)$, then one has
\[
|\langle \gamma_{\Kn}, c\rangle_{L^2(\rd{x})}| \leq \delta\,.
\]
then choose $\sigma_s(x)$ such that $c(x)$ vanishes in the layer,
\[
\langle\gamma_\Kn, c\rangle_{L^2(\rd{x})} = \langle(\gamma_\Kn)_\text{int}, c_\text{int}\rangle_{L^2(\rd{x})}
= \langle -\frac{C\Kn}{\sigma_{s0}^2}\int_0^T \nabla_x \rho_f \cdot \nabla_x \rho_g \rd{t} , ~c_\text{int} \rangle_{L^2(\rd{x})}+ \mathcal{O}(\Kn^2)\,,
\]
we see that
\[
c \sim \mathcal{O} \left( \frac{\delta }{ \Kn}\right) \,.
\]
\end{proof}

%%%%%%%%%%%%%%%%%%%%%%%%%%%%%%%%%%%%%%%%%%%%%%%%%%%%%%
\section{Highly forward peaked regime}
In this section, we consider the anisotropic scattering, and study the well/ill-posedness of the inverse RTE in the highly forward peaked regime, in which the time-dependent RTE is asymptotically equivalent to the Fokker-Planck equation. For simplicity, we study the critical case with zero absorption and $x$-independent scattering. The radiative transfer equation reads
\begin{equation}\label{eqn:hfp}
\begin{cases}
\partial_tf(t,x,v)+ v\cdot \nabla_x f(t,x,v)=\mathcal{L}f(t,x,v)  \,, \\
f(0,x,v)=f^I(x,v)\quad\text{on}\quad \Omega\times \Sp^{d-1}  \,, \\
f(t,x,v) = \phi(t,x,v) \quad\text{on}\quad (0,T)\times \Gamma_- \,,
\end{cases}
\end{equation}
where the collision operator takes the form:
\begin{equation}\label{eqn:hfp_L}
\mathcal{L}f(t,x,v)=\frac{1}{\epsilon^2}\int_{\Sp^{d-1}}\sigma \left(\frac{1-v\cdot v'}{\epsilon} \right)(f(t,x,v')-f(t,x,v))\rd{v'}\,.
\end{equation}
Without loss of generality, we assume that $\sigma$ integrates to one, i.e,
\begin{equation} \label{eqn:uuu}
\frac{1}{\epsilon} \int_{\Sp^{d-1}} \sigma\left( \frac{1- v\cdot v'}{ \epsilon }\right) \rd v' = 1\,.
\end{equation}

Considering $v'$ is the incident direction and $v$ is the scattering direction, then the small parameter $\epsilon$ reinforces ``small-angle" scattering---the kernel is peaked in the forward direction of flight; it also plays a role of mean free path, which accounts for strong scattering effect. Here $v\in \Sp^{d-1}$ is a unit vector denoting the direction of flight. Hereafter, we will focus on dimension $d=3$.

The existence of such a regime was longly exposed to the area~\cite{Pomraning_FP, Larsen_FP}, but has received little attention in the inverse problem setting. It is not quite known how stabilities change according to $\epsilon$ despite some conjectures~\cite{Bal_review}. We address this issue in this section. We will first formally derive the Fokker-Planck limit in section~\ref{sec:FP}, and set up the inverse problem in section~\ref{sec:inverse_FP}. Stability with respect to $\epsilon$ will be discussed in section~\ref{sec:stability_FP}.

%%%%%%%%%%%%%%%%%%%%%%%%%%%%%%%%%%%%%%%%%%%%%%%%%
\subsection{Fokker-Planck Limit}\label{sec:FP}
The equation, in the zero limit of $\epsilon$, loses the large-angle scattering, and effectively is equivalent to the Fokker-Planck equation. The original derivation was seen in~\cite{Pomraning_FP,Larsen_FP}. Denote $\mu=v\cdot v'$ the cosine of the scattering angle, then the scattering cross-section has the Legendre polynomial expansion:
\begin{equation*}
\frac{1}{\epsilon}\sigma \left(\frac{1-\mu}{\epsilon} \right)=\sum_{n=0}^\infty\frac{2n+1}{4\pi}\sigma_{n} P_n(\mu) \,,
\end{equation*}
where the projection coefficients on the $n$-th Legendre polynomial $P_n(\mu)$ is:
\begin{equation} \label{eqn:sigmasn_define}
\sigma_{n} = \frac{2\pi}{\epsilon}\int_{-1}^1\sigma \left(\frac{1-\mu}{\epsilon}\right)P_n(\mu)\rd{\mu}\,.
\end{equation}
It is immediate that $\sigma_{0}=1$ from \eqref{eqn:uuu}.

To proceed, we write $v$ using spherical coordinates: $v = (\sqrt{1-v_3^2}\cos\psi,\sqrt{1-v_3^2}\sin\psi,v_3)$ and introduce the spherical harmonic functions
\begin{equation*}
Y_{n,m}(v)= \left[\frac{2n+1}{4\pi}\frac{(n-|m|)!}{(n+|m|)!} \right]^{1/2}\times (-1)^{(m+|m|)/2}P_{n,|m|}(v_3)e^{in\psi}\,, n \geq0, ~ -n \leq m \leq n \,,
\end{equation*}
where $P_{n,|m|}(v_3)$ are associated Legendre functions
\begin{equation*}
P_{n,|m|}(v_3)=(1-v_3^2)^{m/2} \left(\frac{\rd}{\rd{v_3}} \right)^m P_n(v_3),\quad 0\leq m\leq n\,.
\end{equation*}
The spherical harmonic functions form a complete set of orthonormal basis and thus any suitably smooth function $f(v)$ defined on the unit sphere can be expanded as
\begin{equation}\label{eqn:hfp_exp}
f(v) = \sum_{n=0}^\infty\sum_{m=-n}^n f_{n,m}Y_{n,m}(v), \qquad f_{n,m}:=\int_{\Sp^2}f(v)Y_{n,m}(v)\rd{v}\,.
\end{equation}
Note also that $P_n(\mu)$ satisfy the addition formula:
\[
P_n(v\cdot v') = \frac{4\pi}{2n+1} \sum_{m=-n}^{n} Y_{n,m}(v) Y_{n,m} (v') \,.
\]
Therefore, the collision \eqref{eqn:hfp_L} admits the following expansion
\begin{eqnarray} \label{eqn:Lf00}
\opL f  &= & \frac{1}{\epsilon}  \sum_{n=0}^\infty \sum_{m=-n}^n   \sigma_{n} \int_{\Sp^2}  Y_{n,m}(v) Y_{n,m}(v') \left[ f(v')-f(v) \right] \rd v' \nonumber
\\ & =&  \frac{1}{\epsilon} \left[  \sum_{n=0}^\infty \sum_{m=-n} ^n \sigma_{n} f_{n,m} Y_{n,m} (v)  - \sigma_{0} f(v) \right]  \nonumber
\\&=& \frac{1}{\epsilon} \sum_{n=0}^\infty \sum_{m=-n}^n   (\sigma_{n} - \sigma_{0}) f_{n,m} Y_{n,m}(v) \,.
\end{eqnarray}
The second equality holds because $\int Y_{n,m} \rd{v}= 0$ for all $n\geq 1$ and $\int Y_{n,m} \rd{v}= 1$ only if $n=m=0$.
Let $\alpha = \frac{1-\mu}{\epsilon}$, we rewrite $\sigma_{n}$ in \eqref{eqn:sigmasn_define} as 
\begin{eqnarray} \label{eqn:sigmasn2}
\sigma_{n} &=& 2\pi \int_0^{2/\epsilon} \sigma(\alpha)  P_n(1-\epsilon \alpha ) \rd \alpha  \nonumber
\\ &=& 2\pi \int_0^{2/\epsilon} \sigma (\alpha)  \left[  P_n(1) - P_n'(1) \epsilon \alpha + \frac{P_n''(1)}{2} (\epsilon \alpha)^2 + \cdots     \right] \rd\alpha\,.
\end{eqnarray}
If we define 
\begin{equation}\label{eqn:xi}
\begin{aligned}
\xi_n&:=2\pi \int_{-1}^1 \frac{1}{\epsilon^2}\sigma \left(\frac{1-\mu}{\epsilon} \right)(1-\mu)^n\rd{\mu}\\
&=\epsilon^{n-1}\left(2\pi\int_0^{2/\epsilon}t^n\sigma(t)\rd{t}\right)\\
&=\mathcal{O}(\epsilon^{n-1}) \,,
\end{aligned}
\end{equation}
then $\sigma_{n}$ can be rewritten as 
\begin{equation} \label{eqn:sigmasn}
\sigma_{n} = \epsilon \left[ P_n(1) \xi_0 + P_n'(1) \xi_1 + \frac{1}{2} P_n''(1) \xi_2 + \frac{1}{3!} P_n'''(1) \xi_3 + \cdots    \right]\,.
\end{equation}
Note that $\xi_0$ is fixed and has no dependence on $\sigma$ due to~\eqref{eqn:uuu}:
\begin{equation} \label{sigma_0}
\xi_0=2\pi \int_{-1}^1 \frac{1}{\epsilon^2}\sigma \left(\frac{1-\mu}{\epsilon} \right)\rd{\mu} = \frac{2\pi}{\epsilon}\,.
\end{equation}

Since 
\[
P_n(1) = 1, \quad P_n'(1) = \frac{n(n+1)}{2}, \quad P_0(\mu) = 1\,,
\]
we have from \eqref{eqn:sigmasn2} that 
\begin{eqnarray*}
\sigma_{n} - \sigma_{0} = - \epsilon \frac{n(n+1)}{2} \xi_1  + \mathcal{O} (\epsilon^2) \,,
\end{eqnarray*}
and therefore plugging it into \eqref{eqn:Lf00} we get
\begin{equation} \label{eqn:Lf18}
\opL f = \sum_{n=0}^\infty \sum_{m=-n}^n -\frac{n(n+1)}{2} \xi_1 f_{n,m} Y_{n,m}  + \mathcal{O} (\epsilon) \,.
\end{equation}

Recall that for the well-known Fokker-Planck operator in spherical coordinates:
\begin{equation*}
\opL_\text{FP} f(v) = \left[   \frac{\partial}{\partial v_3} (1-v_3^2) \frac{\partial}{\partial v_3} + \frac{1}{1-v_3^2} \frac{\partial^2}{ \partial \psi ^2}   \right] f(v) \,,
\end{equation*}
we have
\begin{equation} \label{eqn:LFP}
\opL_\text{FP} Y_{n,m}(v) = -n(n+1) Y_{n,m} (v)\,.
\end{equation}
Comparing \eqref{eqn:Lf18} and \eqref{eqn:LFP}, we get the Fokker-Planck approximation: 
\begin{equation*}
\begin{aligned}
\mathcal{L}f(v)&=\sum_{n=0}^\infty\sum_{m=-n}^n -\left(\frac{n(n+1)}{2}\xi_1+\mathcal{O}(\epsilon) \right)f_{n,m}Y_{n,m}
&=\frac{\xi_1}{2}\opL_\text{FP}f(v)+\mathcal{O}(\epsilon)\,,
\end{aligned}
\end{equation*}
with
\begin{equation} \label{eqn:xi1}
\xi_1 = 2\pi\int_0^{\epsilon/2}t\sigma(t)\rd{t}\sim 2\pi\int_0^\infty t\sigma(t)\rd{t}\,.
\end{equation}
In other words, when $\epsilon$ is small, the linear scattering operator $\opL$ converges to the Fokker-Planck operator with a scalar multiplication and the linear transport equation converges to the Fokker-Planck equation
\[
\partial_t f + v \cdot \nabla_x f = \frac{\xi_1}{2} \opL_{\text{FP}} f\,,
\]
where $\opL_{\text{FP}}$ and $\xi_1$ are defined in \eqref{eqn:LFP} and \eqref{eqn:xi1}, respectively. 

\begin{remark}\label{rmk:L_vs_LFP}
We note that the unknown in the collision term defined in~\eqref{eqn:hfp_L} is $\sigma(\mu)$. As a function of $\mu$, it could be fully recovered only if all the coefficients $\sigma_n$ in~\eqref{eqn:sigmasn_define} are known. According to~\eqref{eqn:sigmasn}, this requires knowledge about $\xi_n$ for all $n$. However, in the zero limit, the collision term converges to the Laplace operator, and there is only one scalar that is unknown: $\xi_1$. As a result, the limiting Fokker-Planck equation is much easier to invert heuristically. This will be reflected in section~\ref{sec:stability_FP}.
\end{remark}

%%%%%%%%%%%%%%%%%%%%%%%%%%%%%%%%%%%%
\subsection{Inverse problem setup}\label{sec:inverse_FP}
In the inverse problem setting, we are given inflow data and measure the outflow, with which we infer the scattering coefficient $\sigma(v,v')$. Here the albedo operator is given by:
\begin{equation*}
\mathcal{A}(\sigma): \qquad \phi(t,x,v)|_{(0,T)\times\Gamma_-}\rightarrow \int_{\Gamma_+(y)}f(t,y,v)n(y)\cdot v\rd{v} \,.
\end{equation*}

We first linearize the albedo operator. Like always, we assume that a priori knowledge provides a background state $\sigma_0$ such that the residue $\tilde{\sigma}:=\sigma-\sigma_0$ satisfies
\[
|\tilde{\sigma}| \ll |\sigma| \,,\quad\text{a.s.}\,,
\]
then with background state $\sigma_0$, one gets the solution $f_0$ that solves the following initial boundary value problem 
\begin{equation} \label{eqn:hfp_000}
\begin{cases}
\partial_t f_0(t,x,v)+ v\cdot \nabla_x f_0(t,x,v)=\mathcal{L}_0f_0(t,x,v) \,,
\\ f_0(0,x,v)=0 \quad\text{on}\quad  \Omega\times \mathbb{S}^2 \,,
\\ f_0(t,x,v)=\phi(t,x,v) \quad\text{on}\quad (0,T)\times \Gamma_- \,,
\end{cases} 
\end{equation}
where
\begin{equation*}
\mathcal{L}_0f_0(t,x,v)=\frac{1}{\epsilon^2}\int_{4\pi} \sigma_0 \left(\frac{1-\mu}{\epsilon} \right)(f_0(t,x,v')-f_0(t,x,v))\rd{v'}\,.
\end{equation*}
The residue $$\tilde{f}(t,x,v):=f(t,x,v)-f_0(t,x,v)$$ then satisfies 
\begin{equation}\label{eqn:hfp_tilde}
\partial_t \tilde{f}(t,x,v)+ v\cdot \nabla_x \tilde{f}(t,x,v)=\mathcal{L}_0\tilde{f}(t,x,v)+\tilde{\mathcal{L}}f_0(t,x,v)
\end{equation}
with zero initial data and boundary data. Here 
\begin{equation} \label{eqn:Ltilde}
\tilde{\mathcal{L}}f_0(t,x,v)=\frac{1}{\epsilon^2}\int_{4\pi} \tilde{\sigma} \left(\frac{1-\mu}{\epsilon} \right)(f_0(t,x,v')-f_0(t,x,v) )\rd{v'} \,.
\end{equation}
We also define an adjoint problem to \eqref{eqn:hfp_000}:
\begin{equation}\label{eqn:hfp_g}
\begin{cases}
-\partial_t g(t,x,v)-v\cdot \nabla_x g(t,x,v)=\mathcal{L}_0g(t,x,v) \,, 
\\ g(T,x,v)=0\quad\text{on}\quad \Omega\times \mathbb{S}^2 \,, 
\\ g(t,x,v)=\delta(\tau,y) \quad\text{on}\quad (0,T)\times \Gamma_+ \,.
\end{cases}
\end{equation}
Multiply \eqref{eqn:hfp_tilde} and \eqref{eqn:hfp_g} by $g$ and $\tilde{f}$ respectively, and subtract them, we get, after integrating in  $x$, $v$ and $t$,
\begin{equation} \label{eqn:618}
\int_{\Gamma_+(y)}\tilde{f}(\tau,y,v)n(y)\cdot v \rd{v} =\int_{\Omega\times\mathbb{S}^2}\int_0^Tg(t,x,v)\tilde{\mathcal{L}}f_0(t,x,v)\rd{t}\rd{v}\rd{x} \,,
\end{equation}
where the LHS is the difference between measurement of $f(t,x,v)$ and the computed $f_0(t,x,v)$ at time $\tau$ and position $y\in \partial\Omega$, and we denote it by $b(\tau, y, \phi)$
\[
 b(\tau, y, \phi) = \int_{\Gamma_+(y)}\tilde{f}(\tau,y,v)n(y)\cdot v \rd{v} \,.
\]
The RHS of \eqref{eqn:618} gives a linear function for $\tilde{\sigma}$. In particular, using \eqref{eqn:Ltilde} we have
\begin{equation}\label{eqn:hfp_IBP}
b(\tau, y, \phi )=\frac{1}{\epsilon^2}\int_{\mathbb{S}^2\times\mathbb{S}^2}\tilde{\sigma} \left(\frac{1-\mu}{\epsilon} \right)\gamma_\epsilon(v,v')\rd{v'}\rd{v}\,,
\end{equation}
where
\begin{equation}\label{eqn:hfp_gamma}
\gamma_\epsilon(v,v';\tau,y,\phi):=\int_0^T\int_{\Omega}g(t,x,v)[f_0(t,x,v')-f_0(t,x,v)]\rd{x}\rd{t} \,.
\end{equation}
By varying $\{\tau,y\}$ and $\phi$, one obtains different $g$ and $f_0$, and thus $\gamma_\epsilon$, making~\eqref{eqn:hfp_IBP} a Fredholm operator of first kind with parameters $\{\tau,y,\phi\}$.

As mentioned in Remark~\ref{rmk:wellposedness_ab}, to have a unique recovery of $\tilde{\sigma}$ in $L_p$ space, one needs $\gamma_\epsilon$ expanding the adjoint space $L_q$ (with $\frac{1}{p} + \frac{1}{q} = 1$). The injectivity is beyond the scope of the current paper, and we only discuss the stability in the following section.

%%%%%%%%%%%%%%%%%%%%%%%%%%%%%%%%%%%%%%%%%%%%%%%%%%
\subsection{Stability in the highly forward peaked regime}\label{sec:stability_FP}
In this section we study the stability in the recovery of $\sigma$ in the forward peaked regime. There are two aspects of the problem:

\begin{itemize}
\item[1.] To fully recover $\sigma(1-v\cdot v')$, as mentioned in Remark~\ref{rmk:L_vs_LFP}, one needs all its moments $\sigma_n$, which in turn requires the information of $\xi_n$ for all $n$. However, since $\xi_n$ diminishes at the order of $\epsilon^{n-1}$, obtaining $\xi_n$ is very sensitive to the pollution in the data. Indeed, suppose the data has pollution of order $\delta$, then there are at most $n_0=\log_\epsilon\delta +1$ terms that can be recovered. Now keeping $\delta$ fixed and sending $\epsilon$ to $0$, the number $n_0$ decreases to $1$, meaning that all the higher order information get lost. This is indeed consistent with the view of the singular decomposition ~\cite{ChoulliStefenov,Bal_review}. In that viewpoint, the reconstruction of $\sigma$ relies on the separation of the ballistic component (pure transport) and the scattered components (mainly the single-scattering). In the forward peaked regime, however, the single-scattering concentrates on the original velocity and does not distinguish from the ballistic transport much, making the separation hard, and thus disables the reconstruction. This will be demonstrated in Theorem~\ref{thm:recovery_sigma}.

\item[2.] Nevertheless, in the Fokker-Planck regime, its not the full information $\sigma(1-v\cdot v')$ that matters, but the rescaled one $\frac{1}{\epsilon}\sigma\left(\frac{1-v\cdot v'}{\epsilon} \right)$.  As written, when $\epsilon$ is small, the rescaled $\sigma$ will concentrate around $v\cdot v'=1$ and only a little information is needed to recover its shape. Indeed, according to~\eqref{eqn:sigmasn2},~\eqref{eqn:xi} and~\eqref{eqn:sigmasn}, $\xi_n$ quickly decays to zero, and all of $\sigma_n$ are dominated by the first few $\xi_n$ for certain accuracy. For example, if $\epsilon$ accuracy is needed for $\sigma_n$, one only needs to recover one parameter $\xi_1$. This significantly reduces the amount of measurements needed. In this sense, we find that the inverse problem with highly forward peaked scattering is actually more practice friendly. This is demonstrated in Theorem~\ref{thm:recovery_sigma_ep}.
\end{itemize}

To recover $\frac{1}{\epsilon} \tilde{\sigma}_s \left( \frac{1-v\cdot v'}{\epsilon}\right)$, one simply needs to find all its Legendre coefficients $\tilde{\sigma}_{n}$ in the expansion
\begin{equation}\label{eqn:tilde_sigma_exp}
\frac{1}{\epsilon} \tilde{\sigma} \left( \frac{1-v\cdot v'}{\epsilon}\right) = \sum_{n=0}^\infty \tilde{\sigma}_n P_n (\mu)\,,
\qquad \mu = v \cdot v'\,.
\end{equation}
Using the same expression as in~\eqref{eqn:sigmasn}, one has:
\begin{equation} \label{eqn:exp000}
\tilde{\sigma}_{n} = \eps \left[ P_n(1) \tilde{\xi}_0 + P_n'(1) \tilde{\xi}_1 + \frac{1}{2} P_n''(1) \tilde{\xi}_2 + \frac{1}{3!} P_n'''(1) \tilde{\xi}_3 + \cdots    \right]\,,
\end{equation}
with
\begin{equation}\label{eqn:xi_tilde}
\tilde{\xi}_n = \epsilon^{n-1}  2\pi \int_0^{\epsilon/2} t^n \tsigmas (t) \rd{t} = \mathcal{O}(\epsilon^{n-1})\,.
\end{equation}
Introducing the above relations into \eqref{eqn:hfp_IBP}, we get
\begin{eqnarray*}
b(\tau, y, \phi) &=&  \frac{1}{\epsilon} \sum_{n=0}^\infty \tsigmasn \int_{\Sp^2} P_n(\mu) \gamma_\epsilon(v,v') \rd{v} \rd{v'} 
\\ &=& \sum_{n=0}^\infty \sum_{j=0}^\infty \xi_j P^{(j)}_n (1) \frac{1}{j!} \int_{\Sp^2} P_n(\mu) \gamma_\epsilon(v,v') \rd{v} \rd{v'} 
\\&=& \sum_{j=0}^\infty  \xi_j   \left(  \frac{1}{j!} \sum_{n=0}^\infty  P^{(j)}_n (1) \int_{\Sp^2} P_n(\mu) \gamma_\epsilon(v,v') \rd{v} \rd{v'} \right) \,.
\end{eqnarray*}
Consequently, we obtain the following linear system for $\vec{\xi} = (\xi_1, \xi_2, \cdots ) $
\begin{equation}\label{Axb}
\Amat \vec{\xi} = \vec{b}\,,
\end{equation}
where $\vec{b}$ is a column vector whose size is equal to the number of experiments, and in matrix $\Amat =[a_{ij}]$, the component $a_{ij}$ is determined by 
\[
a_{ij} = \frac{1}{j!} \sum_{n=0}^\infty  P^{(j)}_n (1) \int_{\Sp^2} P_n(\mu) \gamma_\epsilon(v,v') \rd{v} \rd{v'}  \,,
\]
where the subscript $i$ represent experiment $i$ with choice $\tau_i$, $y_i$, $\phi_i$. Recall the expression of $\gamma_\epsilon$ in \eqref{eqn:hfp_gamma}, one has
\begin{eqnarray*}
\int_{\Sp^2} P_n(\mu) \gamma_\epsilon (v,v') \rd{v} \rd{v'} &=& \frac{4\pi}{ 2n+1} \sum_{m=-n}^n \int  Y_{n,m}(v) Y_{n,m}(v') \gamma_\epsilon(v,v') \rd{v} \rd{v'}
\\&=& \frac{4\pi}{ 2n+1} \sum_{m=-n}^n \int  Y_{n,m}(v) Y_{n,m}(v') \left[\int g(v)f_0(v') - g(v) f_0(v) \rd{x} \rd{t} \right] \rd{v} \rd{v'}
\\&=& \frac{4\pi}{ 2n+1} \sum_{m=-n}^n \left[\bar{g}_{n,m} {(\overline{f_0})}_{n,m}  - (\overline{gf_0})_{n,m}\delta_{n,m}\right]\,,
\end{eqnarray*}
where the over-line denotes integration in both $x$ and $t$. Therefore, 
\begin{equation} \label{aij}
a_{ij} = \sum_{n=0}^\infty \frac{P_n^{(j)}}{j!} \frac{4\pi}{2n+1} \sum_{m=-n}^n \left[\bar{g}_{n,m} {(\overline{f_0})}_{n,m}  - (\overline{gf_0})_{n,m} \delta_{n,m} \right]\,.
\end{equation}

\begin{theorem}\label{thm:recovery_sigma_ep}
The recovery of $\frac{1}{\epsilon} \tilde{\sigma}\left(\frac{1-\mu}{\epsilon}\right)$ does not deteriorates as $\epsilon\to0$. More precisely, if we define 
the distinguishability coefficient as 
\begin{equation}\label{eqn:kepsilon}
\kappa_\epsilon = \sup_{\sigma \in \Gamma_\delta} \frac{\|\frac{1}{\epsilon} \sigma\left( \frac{1-\mu}{\epsilon} \right)  - \frac{1}{\epsilon} \tilde{\sigma} \left( \frac{1-\mu}{\epsilon}\right) \|_\infty}{\left\|   \frac{1}{\epsilon}  \tilde{\sigma} \left( \frac{1-\mu}{\epsilon} \right) \right\|_\infty}\,,
\end{equation}
where
\[
\Gamma_\delta = \left\{ \frac{1}{\epsilon} \sigma \left(\frac{1-\mu}{\epsilon} \right): \quad  \sup_{\substack{\forall \|\phi\|_{L^\infty(\Gamma_-)}\leq 1,\\ \forall y\in \partial\Omega, ~ \tau\in [0,T]}}  \left|\frac{1}{\epsilon} \int  \left[ \frac{1}{\epsilon} \sigma \left( \inner \right) - \frac{1}{\epsilon} \tilde{\sigma} \left( \frac{1-\mu}{\epsilon} \right) \right] \gamma_\epsilon(v,v'; \tau, y, \phi) \rd v \rd v'  \right|  \leq \delta  \right\} \,,
\]
then for $\epsilon \ll \delta $
\[
\kappa_\epsilon \sim \mathcal{O} (\delta \epsilon)\,.
\]
\end{theorem}

\begin{proof}
From \eqref{eqn:tilde_sigma_exp}--\eqref{eqn:xi_tilde}, we see that 
\begin{equation*}
\frac{1}{\epsilon} \tilde{\sigma} \left(\frac{1-\mu}{\epsilon} \right) = \epsilon \sum_{k=0}^\infty \left[ \sum_{n=0}^\infty \frac{2n+1}{4\pi} P_n(\mu) \frac{P_n^{(k)}(1)}{k!} \right] \tilde{\xi}_k := \epsilon \sum_{k=0}^\infty h_k(\mu) \tilde{\xi}_k\,,
\end{equation*}
where $\tilde{\xi}_k$ is defined the same as in \eqref{eqn:xi_tilde}. Similarly, $\frac{1}{\epsilon} {\sigma}\left(\frac{1-\mu}{\epsilon}\right)$ has the expansion
\[
\frac{1}{\epsilon} {\sigma} \left(\frac{1-\mu}{\epsilon} \right) = \epsilon \sum_{k=0}^\infty h_k(\mu) {\xi}_k\,.
\]
Note firstly that $\tilde{\xi}_0 = \xi_0$, then for $\frac{1}{\epsilon} {\sigma} (\frac{1-\mu}{\epsilon})\in \Gamma_\delta$, we have
\begin{eqnarray*}
\left| \frac{1}{\epsilon} \int  \left[ \frac{1}{\epsilon} \sigma \left( \inner \right) - \frac{1}{\epsilon} \tilde{\sigma} \left( \frac{1-\mu}{\epsilon} \right) \right] \gamma_\epsilon(v,v'; \tau, y, \phi) \rd v \rd v' \right|
&=&\left| \sum_{k=1}^\infty \int h_k(\mu) (\xi_k - \tilde{\xi}_k) \gamma_\epsilon \rd v \rd v' \right|
\\ & =& \left|  \int h_1(\mu) (\xi_1 - \tilde{\xi}_1) \gamma_\epsilon \rd v \rd v' + \mathcal{O}(\epsilon)  \right| \leq \delta\,,
\end{eqnarray*} 
which implies that, for $\epsilon \ll \delta$,  $|\xi_1 -\tilde{\xi}_1| \leq \mathcal{O}(\delta)$ since $h_1(\mu)$ and $\gamma_\epsilon $ are $\mathcal{O}(1)$. Plugging this result into the definition \eqref{eqn:kepsilon}, one immediately sees that
\[
\kappa_\epsilon = \sup_{\sigma \in \Gamma_\delta } \frac{(\xi_1 - \tilde{\xi}_1)h_1(\mu) + \mathcal{O}(\epsilon^2)}{ \xi_0 h_0(\mu) + \xi_1 h_1(\mu) + \mathcal{O}(\epsilon^2)} \sim \mathcal{O}(\delta \epsilon)\,,
\]
where $\epsilon$ on the right comes from the fact that $\xi_0 \sim \mathcal{O}(\epsilon^{-1})$.
\end{proof}

On the contrary of the above result, if we want to fully recover $\sigma$, then the presence of the small scale $\epsilon$ will make it impossible. Specifically, we have the following theorem.  
\begin{theorem}\label{thm:recovery_sigma}
Suppose $\sigma\in H_k(\rd{\mu})$ (assuming $k$-th regularity in $\sigma$), the recovering $\sigma$ becomes impossible in the limit of $\epsilon \rightarrow 0$ in the sense that if $\delta$ error is allowed in $\vec{b}$ (e.g., measurement error, see \eqref{Axb}), then the error in $\sigma$ will be $\left(\frac{\ln\epsilon}{\ln\delta}\right)^k$.
\end{theorem}
\begin{proof}
Since $\sigma(\mu)\in L_2(\rd{\mu})$, we expand it using Hermite functions:
\begin{equation}\label{eqn:hermite_sigma}
\sigma = \sum_n\frac{1}{n!}\hat{\sigma}_nH_n(\mu)
\end{equation}
with $\hat{\sigma}_n = \langle \sigma\,,H_n\rangle = \int \sigma H_n(\mu)\rd{\mu}$. Here $H_n$ are weighted Hermite functions written as:
\begin{equation*}
H_n(\mu) = p_n(\mu) e^{-\mu^2/2}\,,\quad\text{with}\quad \int p_m(\mu)p_n(\mu)e^{-\mu^2}\rd{\mu} = \delta_{mn}n!\,.
\end{equation*}
Note that other $L_2$ basis functions can be used. Hermite polynomial is only one possible choice.

Meanwhile we recall definition of $\xi$:
\begin{equation*}
\xi_m = 2\pi\epsilon^{m-1}\int_0^{2/\epsilon}\sigma(\mu)\mu^m\rd{\mu}\,.
\end{equation*}

To prove the theorem, we allow $\Delta$ error in $\vec{\sigma}$, and see how much it affects $\vec{\xi}$, and then $\vec{b}$ in the end. Here $\vec{\xi} = [\xi_0,\xi_1,\cdots]'$ and $\vec{\sigma} = [\sigma_0,\sigma_1,\cdots]'$. We first note that $\sigma\in H_k$, and thus its Hermite polynomial coefficients decay algebraically fast. The standard approximation theory from spectral accuracy indicates:
\begin{equation*}
\sigma_n = \mathcal{O}(1/n^k)\,.
\end{equation*}
Suppose we tolerant error up to $\Delta$, then one needs to recover $\sigma_n$ up to at least $n_0 = \Delta^{-1/k}$, and the allowed perturbation in $\sigma_n$ is:
\begin{equation*}
\Delta_{\sigma_n} \leq \Delta \quad\text{for}\quad n = 0\,,\cdots,n_0\,.
\end{equation*}

We now look for explicit relation between $\xi$ and $\sigma_n$. Considering the explicit relation between monomials and the Hermite polynomials
\begin{equation*}
v^m = m!\sum_{k=0}^{m/2}\frac{1}{2^kk!(m-2k)!}p_{m-2k}(v)\,,
\end{equation*}
plugging it back in the equation for $\xi_m$, we have
\begin{align*}
\xi_m &= \epsilon^{m-1}\sum_{k=0}^{m/2} \frac{ 2 \pi m!}{2^kk!(m-2k)!}\int_0^{2/\epsilon}\sigma(v)p_{m-2k}\rd{v} \\\
&=\epsilon^{m-1}\sum_{k=0}^{m/2}\sum_{n=0}^\infty\frac{2 \pi m!}{2^kk!(m-2k)!n!}D_{n,m,k}\sigma_n \,,
\end{align*}
where we have used expansion in~\eqref{eqn:hermite_sigma} and defined $D_{n,m,k} = \int p_{m-2k}p_ne^{-v^2/2}\rd{v}$\,. In a matrix form one has:
\begin{equation*}
\vec{\xi} = \mathsf{C}\cdot\vec{\sigma} \,,
\end{equation*}
where $\mathsf{C}$ defined by:
\begin{equation*}
\mathsf{C}_{mn} = \epsilon^{m-1} \frac{2\pi m!}{n!}\sum_{k=0}^{m/2}\frac{D_{n,m,k}}{2^kk!(m-2k)!}\,.
\end{equation*}

Since we need to recover $\sigma_n$ up to $n = n_0 = \Delta^{-1/k}$, and the recovered coefficients need to be within error tolerance $\delta$, the tolerance for $\xi_{n_0}$ then is:
\begin{equation*}
\Delta_{\xi_{n_0}} \leq \Delta\sum_{k=0}^{n_0}\mathsf{C}_{n_0,k} \sim \epsilon^{n_0-1}\Delta\,,\quad\text{with}\quad n_0 = \Delta^{-1/k}\,.
\end{equation*}
Noting the relationship between $\xi$ and $\vec{b}$ in \eqref{Axb}, we see that the error allowance on $\vec{b}$ is $\epsilon^{n_0-1}\Delta$. With shrinking $\epsilon$, this restriction becomes more and more severe, making the inverse problem less practical. More specifically if we have $\delta$ error in $\vec{b}$, then setting:
\begin{equation*}
\epsilon^{\Delta^{-1/k}}\Delta \sim \delta
\end{equation*}
gives $\Delta>\left(\frac{\ln\epsilon}{\ln\delta}\right)^k \to\infty$ as $\epsilon\to0$, meaning that the accuracy in recovering $\sigma_n$ is lost and so it is with $\sigma$.
\end{proof}

%%%%%%%%%%%%%%%%%%%%%%%%%%%%%%%%%%%%%%%%%%%%%%%

\bibliographystyle{siam}
\bibliography{inverse_RTE_ref}

\end{document}